\documentclass{amsart}

\newtheorem{thmnum}{}

\ifx\MyBeamer\undefined
\usepackage[breaklinks,colorlinks,pagebackref]{hyperref}
\usepackage[a4paper,margin=1in]{geometry}
\else
\if\MyBeamer t
\usepackage[breaklinks,colorlinks,backref]{hyperref}
\usepackage[a4paper,margin=1in]{geometry}
\usepackage{beamerarticle}
\fi\fi

\usepackage{amsmath}
\usepackage{amssymb}
\usepackage{mathrsfs}
\usepackage{amsthm}
\usepackage{aliascnt}
\usepackage[utf8]{inputenc}
\usepackage[T1]{fontenc}
\ifx\FrenchText\undefined
\usepackage[british]{babel}
\else
\usepackage[british,french]{babel}
\AtBeginDocument{%
\catcode`\:=12%
\catcode`\!=12%
}
\fi

\makeatletter

\newcounter{quotethmcnt}

\ifx\MyBeamer\undefined

\def\equationautorefname~#1\null{(#1)}
\def\itemautorefname~#1\null{#1}
\fi

\ifx\MyBeamer\undefined

\ifx\thmnum\undefined
\newtheorem{thmnum}{}[section]
\fi

\newcommand{\mynewthm}[3][]{%
  \newaliascnt{#2}{thmnum}%
  \newtheorem{#2}[#2]{#3}%
  \aliascntresetthe{#2}%
  \newtheorem*{#2*}{#3}%
  \expandafter\newcommand\csname #2autorefname\endcsname{#3}%
  \expandafter\renewcommand\csname the#2\endcsname{\thethmnum}%
}

\newtheorem*{clm}{Claim}
\newenvironment{clmprf}{%
  \begin{proof}[Proof of claim]%
  }{\end{proof}}

\else

\let\xxx=\frametitle
\def\frametitle#1{%
  \xxx{%
    \setbeamercolor*{math text}{use={titlelike,my math text},fg=titlelike.fg!80!my math text.fg}%
    #1}%
  \setbeamercolor{math text}{use=my math text,fg=my math text.fg}%
}

\newcommand{\beamerenv}[3]{%
\newenvironment<>{#1}%
{%
  \setbeamercolor{temp}{fg=structure.fg}%
  \setbeamercolor{structure}{fg=#2}%
  \setbeamercolor{block body}{use=structure,bg=structure.fg!5!white}%
  \begin{#3}%
}%
{\end{#3}\setbeamercolor{structure}{fg=temp.fg}}}

\newcommand{\mynewthm}[3][green!50!black]{%
  \newtheorem*{#2x}{#3}%
  \beamerenv{#2}{#1}{#2x}%
}

\beamerenv{other}{violet}{block}

\fi

\ifx\FrenchText\undefined
\newcommand{\myiffrench}[2]{#2}
\else
\newcommand{\myiffrench}[2]{\iflanguage{french}{#1}{#2}}
\fi

\theoremstyle{plain}
\mynewthm[red]{thm}{\myiffrench{Théorème}{Theorem}}
\mynewthm{prp}{Proposition}
\mynewthm[purple]{lem}{\myiffrench{Lemme}{Lemma}}
\mynewthm[orange]{fct}{\myiffrench{Fait}{Fact}}
\mynewthm[purple!50!white]{cor}{\myiffrench{Corollaire}{Corollary}}

\theoremstyle{definition}
\mynewthm[green!80!black]{dfn}{\myiffrench{Définition}{Definition}}
\mynewthm{hyp}{\myiffrench{Hypothèse}{Hypothesis}}
\mynewthm{conv}{Convention}
\mynewthm{conj}{Conjecture}
\mynewthm{ntn}{Notation}
\mynewthm{cst}{Construction}

\theoremstyle{remark}
\mynewthm[yellow!90!black]{rmk}{\myiffrench{Remarque}{Remark}}
\mynewthm{qst}{Question}
\mynewthm{exc}{\myiffrench{Exercice}{Exercise}}
\mynewthm{exm}{\myiffrench{Exemple}{Example}}

\ifx\MyBeamer\undefined

\newcommand{\myenumlabel}[1]{\textnormal{(\roman{#1})}}

\fi

\newcounter{cycprfcnt}
\newcounter{cycprffirst}
\newcommand{\cycprfpreamble}[1]%
{%
  \setcounter{cycprfcnt}{1}
  \setcounter{cycprffirst}{#1}
  \setlength{\itemindent}{0.5\leftmargin}%
  \setlength{\leftmargin}{0pt}%
  \newcommand{\cpcurr}{\myenumlabel{cycprfcnt}}%
  \newcommand{\cpnext}{\addtocounter{cycprfcnt}{1}\cpcurr}%
  \newcommand{\impnext}{\cpcurr{} $\Longrightarrow$ \cpnext.}%
  \def\makelabel##1{\ifnum\value{cycprffirst}=0\hspace{-0.7\itemindent}\setcounter{cycprffirst}{1}\fi##1}%
}%

{\begin{list}{\impnext}{\cycprfpreamble{#1}}}%
{\qedhere\end{list}}%

\newenvironment{cycprf*}[1][0]%
{\begin{list}{\impnext}{\cycprfpreamble{#1}}}%
{\end{list}}%

\def\indsym#1#2{%
  \setbox0=\hbox{$\m@th#1x$}%
  \kern\wd0%
  \hbox to 0pt{\hss$\m@th#1\mid$\hbox to 0pt{$\m@th#1^{#2}$\hss}\hss}%
  \lower.9\ht0\hbox to 0pt{\hss$\m@th#1\smile$\hss}%
  \kern\wd0}

\def\nindsym#1#2{%
  \setbox0=\hbox{$\m@th#1x$}%
  \kern\wd0%
  \hbox to 0pt{\hss$\m@th#1\not$\kern1.4\wd0\hss}
  \hbox to 0pt{\hss$\m@th#1\mid$\hbox to 0pt{$\m@th#1^{#2}$\hss}\hss}%
  \lower.9\ht0\hbox to 0pt{\hss$\m@th#1\smile$\hss}%
  \kern\wd0}

\def\dotminussym#1#2{%
  \setbox0=\hbox{$\m@th#1-$}%
  \kern.5\wd0%
  \hbox to 0pt{\hss\hbox{$\m@th#1-$}\hss}%
  \raise.6\ht0\hbox to 0pt{\hss$\m@th#1.$\hss}%
  \kern.5\wd0}

\renewcommand{\setminus}{\smallsetminus}

\DeclareMathOperator{\Hom}{Hom}

\DeclareMathOperator{\Frac}{Frac}

\DeclareMathOperator{\rk}{rk}

\DeclareMathOperator{\st}{st}

\newcommand{\cM}{\mathcal{M}}
\newcommand{\cN}{\mathcal{N}}

\newcommand{\cU}{\mathcal{U}}

\newcommand{\bC}{\mathbf{C}}

\newcommand{\bR}{\mathbf{R}}

\makeatother

\begin{document}

\title{Tensor products of valued fields}

\author{Itaï \textsc{Ben Yaacov}}

\address{Itaï \textsc{Ben Yaacov} \\
  Université Claude Bernard -- Lyon 1 \\
  Institut Camille Jordan, CNRS UMR 5208 \\
  43 boulevard du 11 novembre 1918 \\
  69622 Villeurbanne Cedex \\
  France}

\urladdr{\url{http://math.univ-lyon1.fr/~begnac/}}

\thanks{Research supported by the Institut Universitaire de France and ERC Grant no.\ 291111.}


\begin{abstract}
  We give a short argument why the tensor product semi-norm on $K \otimes_k L$ is multiplicative when $k$ is an algebraically closed valued field and $K$ and $L$ are valued extensions (valued in $\bR$).
  When the valuation on $k$ is non trivial we use the fact that $ACVF$, the theory of algebraically closed (non trivially) valued fields, has quantifier elimination.
\end{abstract}

\maketitle

It is a classical fact (e.g., Zariski and Samuel \cite[Chapter~IV, Theorem~40, Corollary~1]{Zariski-Samuel:CommutativeAlgebra}) that any two extensions $K$ and $L$ of an algebraically closed field $k$ the ring $K \otimes_k L$ is an integral domain (and this characterises algebraically closed fields).
When $K$ and $L$ (and therefore $k$) are valued in $(\bR^{\geq 0},\cdot)$, the tensor product carries a natural semi-norm; by analogy with the non valued case, if $k$ is also algebraically closed, we would expect this norm to be ``prime'', i.e., multiplicative, extending to a valuation of the fraction field.
This is indeed proved by Jérôme \textsc{Poineau} \cite[Corollaire~3.14]{Poineau:Angelique}, with both the result and the proof stated in the language of Berkovich spaces, making them fairly obscure to those not familiar with this formalism (such as the author, who is indebted to Amaury \textsc{Thuillier} for having pointed to and explained Poineau's result).
Here we propose a more direct proof, using quantifier elimination for the theory of algebraically closed valued fields.

\begin{dfn}
  \label{dfn:ValuedField}
  A \emph{valued field} is a pair $k = (k,O^k) = (k,O)$ where $k$ is a field and the \emph{valuation ring} $O \subseteq k$ is a sub-ring such that $k = O \cup \bigl( O \setminus \{0\} \bigr)^{-1}$.
  We let $O^\times$ denote the group of units of $O$, and call $(\Gamma^k,\cdot) = (\Gamma,\cdot) = k^\times/O^\times$ the \emph{value group}.
  We let $|{\cdot}| \colon k^\times \rightarrow \Gamma$ denote the quotient map, and add a formal symbol $0 = |0|$.
  We order $\Gamma \cup \{0\}$ by $|a| \leq |b| \Longleftrightarrow a \in bO$.
  By a \emph{standard valued field} we mean a valued field together with an embedding $(\Gamma,\cdot,<) \rightarrow (\bR^{>0},\cdot,<)$.

  An \emph{embedding} of valued fields must respect the valuation ring (in both directions), and therefore induces an embedding of the value groups.
  An embedding of standard valued fields is also required to respect the embeddings of the value groups in the reals.

  A \emph{semi-normed $k$-vector space} (or $k$-module) is a $k$-vector space $U$ together with a \emph{semi-norm} $\|{\cdot}\|\colon U \rightarrow \Gamma^U \cup \{0\}$, where $\Gamma^U$ is an ordered multiplicative group extending $\Gamma^k$, satisfying $\|x+y\| \leq \max \|x\|, \|y\|$ and $\|ax\| \leq |a| \|x\|$ (whence $\|ax\| = |a| \|x\|$) for $x,y \in U$, $a \in k$.
  It is \emph{standard} if $k$ is standard and $U$ is equipped with an embedding $\Gamma^U \rightarrow (\bR^+,\cdot,<)$ extending that of $\Gamma^k$.
  In particular, a (standard) valued field extending $k$ is a (standard) normed $k$-vector space.
\end{dfn}

We refer the reader to any standard textbook on model theory, e.g., Poizat \cite{Poizat:Cours} for a general discussion of structures, quantifier elimination, elementary extensions and ultra-powers.

\begin{fct}[{\cite[Theorem~2.1.1(i)]{Haskell-Hrushovski-Macpherson:EliminationOfImaginariesInValuedFields}}]
  \label{fct:QE}
  Consider a valued field as a logical structure in the language of fields (i.e., of rings), together with a predicate symbol for the binary relation $|x| \leq |y|$.
  Then the theory of algebraically closed non trivially valued fields (commonly denoted $ACVF$) in this language has quantifier elimination.
  In particular, if $K/k$ is an extension of such fields then $K \succeq k$ is an elementary extension.
\end{fct}

\begin{fct}
  \label{fct:UltraPowerEmbedding}
  If $\cM$ is any structure, in the sense of first order logic (e.g., a valued field, or a pair of a valued field and a sub-field) and $\cU$ is an ultra-filter then the ultra-power $\cM^\cU$ is an elementary extension of $\cM$.
  Conversely, every elementary extension $\cN \succeq \cM$ embeds over $\cM$ in some ultra-power of $\cM$.
\end{fct}

When $K/k$ is a field extension, let $\langle \ldots \rangle_k$ denote the span in $K$ viewed as a $k$-vector space.

\begin{lem}
  \label{lem:UltraPower}
  Let $K/k$ be an extension of valued fields.
  \begin{enumerate}
  \item
    \label{item:UltraPowerPair}
    For any ultra-filter $\cU$, the embeddings $k \subseteq k^\cU \subseteq K^\cU$ and $k \subseteq K \subseteq K^\cU$ yield a commutative diagram.
    Given $X \subseteq K$, $y \in K$, and $\gamma \in \Gamma^K$ such that $|y| \leq \gamma |y'|$ for all $y' \in y + \langle X \rangle_k$, we also have $|y| \leq \gamma |y'|$ for all $y' \in y + \langle X \rangle_{k^\cU}$.
  \item
    \label{item:UltraPowerExtensionEmbedding}
    If $k$ is algebraically closed and non trivially valued, then there exists an ultra-filter $\cU$ and an embedding $\iota\colon K \rightarrow k^\cU$ which is the identity on $k$.
  \item
    \label{item:UltraPowerStandardPart}
    If $K/k$ in the previous item is moreover an extension of standard valued fields then the embedding $\Gamma^k \subseteq \bR^{>0}$ induces $\Gamma^{(k^\cU)} = (\Gamma^k)^\cU \subseteq (\bR^\cU)^{>0}$, and $|a| = \st |\iota a|$ for all $a \in K^\times$, where $\st$ denotes the standard part map (so in particular, $|\iota a| \in \bR^\cU$ lies in the convex hull of $\bR^+$).
  \end{enumerate}
\end{lem}
\begin{proof}
  For the first assertion, we use the fact that the pair $(K,k)^\cU = (K^\cU,k^\cU)$ is an elementary extension of $(K,k)$.
  For the second assertion, we may assume that $K$ is algebraically closed.
  By \autoref{fct:QE} we have $K \succeq k$, and we may conclude using \autoref{fct:UltraPowerEmbedding}.
  For the last assertion, for $a \in k^\times$ we have $\st |\iota a| = |\iota a| = |a|$.
  Choose $a \in k^\times$ such that in addition $|a| > 1$, and let $b \in K^\times$.
  Let $m,n$ be integers such that $\frac{m}{n} \leq \log_{|a|} |b| \leq \frac{m+1}{n}$ and $n > 0$.
  Then $|a^m| \leq |b^n| \leq |a^{m+1}|$, so $|a|^{m/n} \leq |\iota b| \leq |a|^{(m+1)/n}$ and therefore $|a|^{m/n} \leq \st |\iota b| \leq |a|^{(m+1)/n}$ as well.
  Our assertion follows.
\end{proof}

\begin{ntn}
  Tuples, e.g., $\bar a = (a_0, \ldots, a_{m-1})$, are always indexed starting at zero.
  We shall consider tuples as column vectors, so $\bar a^t \cdot \bar b = \sum_i a_ib_i$, and by analogy, $\bar x^t \otimes \bar y = \sum x_i \otimes y_i$ for tensors.
\end{ntn}

\begin{dfn}[{\cite[2.1.7]{Bosch-Guentzer-Remmert:NonArchimedeanAnalysis}}]
  \label{dfn:TensorProductNorm}
  Let $k$ be a standard valued field, $U$ and $V$ two standard semi-normed vector spaces over $k$.
  For $z \in U \otimes_k V$ we define
  \begin{gather}
    \label{eq:TensorProductNorm}
    \|z\| = \inf_{z = \bar x^t \otimes \bar y} \ \max_i \|x_i\| \|y_i\|.
  \end{gather}
\end{dfn}

There is a canonical embedding $U \otimes_k V \subseteq \Hom_k(U^*,V)$, and let $\rk z$ denote the (finite) rank of $z$ as a morphism $U^* \rightarrow V$.
A presentation $z = \bar x^t \otimes \bar y$ has length $\rk z$ if and only if it has minimal length if and only if each of the tuples $\bar x$ and $\bar y$ is linearly independent over $k$.

\begin{lem}
  \label{lem:TensorProductNorm}
  With the hypotheses of \autoref{dfn:TensorProductNorm}:
  \begin{enumerate}
  \item The function $z \mapsto \|z\|$ defined in \autoref{eq:TensorProductNorm} is a semi-norm on $U \otimes_k V$.
    If $U$ and $V$ are $k$-algebras (commutative, or at least such that $k$ is central) with sub-multiplicative semi-norms (namely, $\|xy\| \leq \|x\| \|y\|$) then the tensor product semi-norm is sub-multiplicative as well.
  \item We may restrict \autoref{eq:TensorProductNorm} to presentations of $z$ of length $\rk z$ without changing the result.
  \item We have $\|x \otimes y\| = \|x\| \|y\|$.
  \end{enumerate}
\end{lem}
\begin{proof}
  The first item is immediate (see \cite{Bosch-Guentzer-Remmert:NonArchimedeanAnalysis}).
  For the second, consider a presentation $z = \bar x^t \otimes \bar y$ of length $m+1$, with $\bar x$ linearly dependent over $k$.
  Up to a permutation we may assume that $x_m = \sum_{i<m} a_i x_i$ where $a_i \in k$ and $\|x_m\| = \max_{i<m} \|a_i x_i\|$.
  Then $z = \sum_{i<m} x_i \otimes y_i'$, where $y_i' = y_i + a_i y_m$, and
  \begin{gather*}
    \max_{i\leq m} \|x_i\|\|y_i\|
    = \max_{i<m} \max\bigl( \|x_i\|\|y_i\|, \|a_i x_i\|\|y_m\| \bigr)
    \geq  \max_{i<m} \|x_i\|\|y_i'\|.
  \end{gather*}
  The second item follows.
  The third item follows from the second.
\end{proof}

(If $U$ and $V$ are standard normed and $k$ complete then $U \otimes_k V$ is normed, but for our purposes this is beside the point.)

\begin{lem}
  \label{lem:ValueEstimate}
  Let $k$ be a valued field, $U$ a semi-normed $k$-vector space, $\Gamma = \Gamma^U$, and let $\bar x \in U^m$ and $\bar \gamma \in \Gamma^m$ be such that:
  \begin{gather}
    \label{eq:ValueEstimateCondition}
    \tag{$*_{\bar \gamma,k}$}
    \|x_i\| \leq \gamma_i \|x\| \quad \text{for all } i < m \text{ and } x \in x_i + \langle x_{<i} \rangle_k.
  \end{gather}
  \begin{enumerate}
  \item
    \label{item:ValueEstimateField}
    For every $\bar a \in k^m$:
    \begin{gather*}
      \|\bar a^t \cdot \bar x\| \prod \gamma_i \geq \max_i |a_i| \|x_i\|.
    \end{gather*}
  \item
    \label{item:ValueEstimateTensorProduct}
    Assume in addition that $k$ is standard and $U$ and $V$ are standard semi-normed $k$-vector spaces.
    Then for every $\bar y \in V^m$:
    \begin{gather*}
      \|\bar x^t \otimes \bar y\| \prod \gamma_i \geq \max_i \|x_i\| \|y_i\|.
    \end{gather*}
  \end{enumerate}
\end{lem}
\begin{proof}
  Dropping those $x_i$ such that $\|x_i\| = 0$ and the corresponding $a_i$ or $y_i$ will not change either the hypotheses or the conclusions.
  We may therefore assume that $\|x_i\| > 0$ for all $i$, in which case the hypothesis \autoref{eq:ValueEstimateCondition} implies that $\bar x$ is linearly independent over $k$ and $\gamma_i \geq 1$ for all $i$.
  For $0 \leq j \leq m$ let $\beta_j = \prod_{i < j} \gamma_i$.

  For the first assertion, let $\alpha = \max_i |a_i| \|x_i\|$.
  Then $\alpha = \|a_i x_i\| \leq \beta_i \|a_i x_i\|$ for some $i$, and we may choose $\ell < m$ maximal such that $\alpha \leq \beta_\ell \|a_\ell x_\ell\|$.
  By \autoref{eq:ValueEstimateCondition} we have $\beta_{\ell+1} \|\sum_{i \leq \ell} a_i x_i\| \geq \beta_\ell \| a_\ell x_\ell \| \geq \alpha$.
  By choice of $\ell$ we have $\beta_{\ell+1} \|a_i x_i\| \leq \beta_i \|a_i x_i\| < \alpha$ for $i > \ell$, so $\beta_m \|\bar a^t \cdot \bar x\| \geq \beta_{\ell+1} \|\bar a^t \cdot \bar x\| \geq \alpha$.

  For the second assertion we need to show that if $\bar x^t \otimes \bar y = z = \bar x'^t \otimes \bar y'$ then $\beta_m \max \|x'_i\| \|y'_i\| \geq \max_i \|x_i\| \|y_i\|$.
  By \autoref{lem:TensorProductNorm} we may assume that $\bar y'$ is linearly independent over $k$.
  Then $\bar x' \subseteq \langle \bar x \rangle_k$ (otherwise there is a linear functional $\lambda \in U^*$ which vanishes on $\bar x$ but not on $\bar x'$, and $\lambda \cdot z \in V$ is both zero and non zero), so let us write $x_i' = \sum_{j<m} a_{ij} x_j$.
  By linear independence of $\bar x$, a tensor calculation yields, $y_j = \sum_{i<n} a_{ij} y'_i$.
  Thus by \autoref{item:ValueEstimateField}
  \begin{gather*}
    \beta_m \max_i \|x_i'\| \|y_i'\| \geq \max_{i,j} |a_{ij}| \|x_j\| \|y_i'\| \geq \max_j \|x_j\| \|y_j\|,
  \end{gather*}
  as desired.
\end{proof}

We can now prove our main result.
The argument goes roughly as follows.
We consider an extension of algebraically closed standard valued fields $K / k$.
As we saw earlier, an ultra-power $K^\cU$ contains two extensions of $k$, namely $K$ and $k^\cU$, and by \autoref{lem:ValueEstimate} the vector space they generate in $K^\cU$ is canonically isometric to $K \otimes_k k^\cU$.
In particular, since we may also embed $L/k$ in $k^\cU/k$ (this is where quantifier elimination is used) we obtain an isometric embedding of the algebra $K \otimes_k L$ in $K^\cU$, and we know that the latter carries a multiplicative valuation.
When $k$ carries a trivial valuation but neither $K$ nor $L$ do, the last step fails, and a completely different argument is required.
While the latter case is uninteresting to us, it is included for the sake of completeness.

\begin{thm}
  \label{thm:MainTheorem}
  Let $k$ be a standard algebraically closed valued field, let $K,L \supseteq k$ be standard valued field extensions, and let $A = K \otimes_k L$.
  Then the tensor semi-norm $\|{\cdot}\|$ on $A$ is multiplicative, extending to a standard valuation on $F = \Frac(A / \ker \|{\cdot}\|)$.
  This renders $F$ an extension of both $K$ and $L$.
\end{thm}
\begin{proof}
  Assume first that the valuation on $k$ is not trivial, and let $\iota\colon L \hookrightarrow k^\cU$ be an embedding as per \autoref{lem:UltraPower}\ref{item:UltraPowerExtensionEmbedding}.
  Since $K$ is a sub-field of $K^\cU$, the universal property of tensor products gives rise to a natural map $\iota\colon A \rightarrow K^\cU$ such that $\|z\| \geq \st |\iota z|$ for all $z \in A$.
  For the converse inequality, let $1 < \gamma \in \Gamma^k$ be rational.
  We may always express $z$ as $\bar x^t \otimes \bar y$, say of length $m$, where $\bar x$ satisfies $(*_{(\gamma,\gamma,\ldots),k})$, and by \autoref{lem:UltraPower}\ref{item:UltraPowerPair}, also $(*_{(\gamma,\gamma,\ldots),k^\cU})$.
  Thus, by \autoref{lem:ValueEstimate}\autoref{item:ValueEstimateField} applied to $K^\cU$ as a normed $k^\cU$-vector space,
  \begin{gather}
    \label{eq:MainTheoremIneqiality}
    \gamma^m \st |\iota z| \geq \st \max_i |x_i \iota y_i| = \max_i |x_i| |y_i| \geq \|z\|.
  \end{gather}
  Since $k$ is algebraically closed, we can choose $1 < \gamma \in \Gamma^k \subseteq \bR$ arbitrarily close to one.
  Then \autoref{eq:MainTheoremIneqiality} gives $\st | \iota z | \geq \|z\|$ and therefore $\st | \iota z | = \|z\|$.
  Thus, for every $z,z' \in A$ we have
  \begin{gather*}
    \|zz'\| = \st | \iota(zz')| = \st | \iota z | \st |\iota z'| = \|z\| \|z'\|,
  \end{gather*}
  as desired.

  When $k$ is trivially valued we need a different argument.
  Call $z \in A$ \emph{$(\alpha,\beta)$-pure} if it can be written as $\bar x^t \otimes \bar y$ with $|x_i| = \alpha$ and $|y_i| = \beta = \|z\|/\alpha$ for all $i$.
  When $z,z' \in A$ are pure, we can multiply them by elements of $K$ and $L$ to reduce to the case where both are $(1,1)$-pure, in which case $\|zz'\| = \|z\| \|z'\|$ holds since the tensor product of the residue fields is an integral domain.

  Say that a presentation $z = \bar x^t \otimes \bar y$ is \emph{normalised} if $x_i$ has least value in $x_i + \langle x_{<i} \rangle$ for each $i$: by \autoref{lem:ValueEstimate}\ref{item:ValueEstimateTensorProduct} we then have $\| \bar x^t \otimes \bar y \| \cdot 1 \geq \max_i |x_i| |y_i|$, i.e., $\|z\| = \max_i |x_i| |y_i|$.
  Since the valuation on $k$ is trivial, if $V \subset K$ is a finite-dimensional $k$-vector space, the valuation takes at most $\dim V + 1$ possible values on $V$.
  It follows that for any presentation $z = \bar x^t \otimes \bar y$ there exists a normalised presentation $z = \bar u^t \otimes \bar v$ where $u_i \in x_i + \langle x_{<i} \rangle$.
  In other words, $\bar u = C^t \bar x$ where $C$ is unipotent upper triangular and $\bar v = C^{-1} \bar y$.
  Say then that $C$ \emph{normalises} $z = \bar x^t \otimes \bar y$.

  Normalising and reordering, any $z \in A$ can be split as $z_0 + z_1$ where $z_0$ is $(\alpha,\beta)$-pure for some $\alpha,\beta$, and $z_1 \in A_{\alpha,\beta} = \Bigl \langle x \otimes y : \bigl( |x| |y|, |y| \bigr) < (\alpha\beta,\beta) \text{ in lexicographic order} \Bigr\rangle$.
  Conversely, we claim that if $z = z_0 + z_1$ where $z_0$ is $(\alpha,\beta)$-pure and $z_1 \in A_{\alpha,\beta}$, as witnessed by $z_0 = \bar x^t \otimes \bar y$ and $z_1 = \bar x'^t \otimes \bar y'$, then $\|z\| = \alpha \beta$.
  We may assume that $|x'_j| |y'_j| = \alpha \beta$, and therefore $|y'_j| < \beta$, for all $j$.
  Let $\begin{pmatrix} C & E \\ 0 & D\end{pmatrix}$ normalise $z = z_0 + z_1 = \bar x^t \otimes \bar y + \bar x'^t \otimes \bar y'$, where the blocks correspond to the two parts, so $C$ normalises $z_0 = \bar x^t \otimes \bar y$.
  Thus $z_0 = \bar u^t \otimes \bar v$ and $z = \bar u^t \otimes \bar w + \bar u'^t \otimes \bar v'$, where $\bar u = C^t \bar x$, $\bar v = C^{-1} \bar y$ and
  \begin{gather*}
    \begin{pmatrix} \bar u \\ \bar u' \end{pmatrix}
    = \begin{pmatrix} C^t & 0 \\ E^t & D^t \end{pmatrix} \begin{pmatrix} \bar x \\ \bar x' \end{pmatrix},
    \qquad
    \begin{pmatrix} \bar w \\ \bar v' \end{pmatrix}
    = \begin{pmatrix} C^{-1} & F \\ 0 & D^{-1} \end{pmatrix} \begin{pmatrix} \bar y \\ \bar y' \end{pmatrix}
    = \begin{pmatrix} C & E \\ 0 & D \end{pmatrix}^{-1} \begin{pmatrix} \bar y \\ \bar y' \end{pmatrix},
  \end{gather*}
  with $F = -C^{-1} E D^{-1}$.
  In particular, $\bar w = \bar v + F \bar y'$.
  Since $\|z_0\| = \alpha \beta$ there is at least one $i$ such that $|u_i| = \alpha$ and $|v_i| = \beta$.
  Since $|w_i - v_i| < \beta$ by our assumption that $|y'_j| < \beta$ for all $j$, we obtain $|w_i| = \beta$ as well and so $\|z\| = \alpha \beta$.

  Now consider $z,z' \in A$, and decompose them $z = z_0 + z_1$ and $z' = z'_0 + z'_0$, where $z_0$ is $(\alpha,\beta)$-pure, $z'_0$ is $(\gamma,\delta)$-pure and $z_1 \in A_{\alpha,\beta}$, $z'_1 \in A_{\gamma,\delta}$.
  By the case of product of two pure elements, $z_0 z'_0$ is $(\alpha\gamma,\beta\delta)$-pure, and clearly $z_1 z'_0 + z_0 z'_1 + z_1 z'_1 \in A_{\alpha\gamma,\beta\delta}$, so $\|zz'\| = \alpha\beta\gamma\delta = \|z\|\|z'\|$, as desired.
\end{proof}

When $k$ is trivially valued and so is one of $K$ or $L$, then a variant of the first argument does go though.
Indeed, if $L$ is trivially valued then an embedding $\iota\colon L \rightarrow k^\cU$ (of pure fields, or of trivially valued fields) exists by quantifier elimination for $ACF$ (the theory of algebraically closed fields), and the rest of the argument remains the same.

\begin{rmk}
  Our definitions only allow for non Archimedean valued fields.
  More generally, an \emph{absolute value} on a field $k$ is a map $|{\cdot}|\colon k \rightarrow \bR^{\geq 0}$, satisfying $|ab| = |a| |b|$, $|a+b| \leq |a|+|b|$, $|0| = 0$ and $|1| = 1$.
  It is a standard fact (e.g., Artin \cite{Artin:AlgebraicNumbersAndFunctions}) that an absolute value is either \emph{Archimedean}, i.e., $|a| = |\iota^k a|^\alpha$, where $\alpha = \log_2 |2| \in (0,1]$ and $\iota^k \colon k \rightarrow \bC$ is uniquely determined up to complex conjugation, or is a standard valuation as defined here.
  In particular, if $K/k$ is an extension of valued fields in this sense, then one is Archimedean if and only if the other is, in which case we may choose $\iota^K$ so that $\iota^k \subseteq \iota^K$, and if $k$ is algebraically closed (or merely such that the image of $\iota^k$ is not contained in $\bR$), this determines $\iota^K$.
  When $K$ and $L$ are two extensions of an algebraically closed Archimedean valued field $k$, we can define on $A = K \otimes_k L$:
  \begin{gather*}
    \|\bar x^t \otimes \bar y\| = | \iota^K \bar x^t \cdot \iota^L \bar y |.
  \end{gather*}
  This is clearly multiplicative, inducing an Archimedean absolute value on $\Frac(A/\ker \|{\cdot}\|)$.
\end{rmk}

\providecommand{\bysame}{\leavevmode\hbox to3em{\hrulefill}\thinspace}

\end{document}